\documentclass[11pt,twoside]{article}
\usepackage{amsfonts}
\usepackage{amsmath}
{\markboth{\sl  C. Abdelkefi and F.
Rached}{\sl
 Further results for the Dunkl Transform and the Ces\`{a}ro operator}
\newtheorem{theorem}{Theorem}[section]
\newtheorem{remark}{Remark}[section]
\newtheorem{example}{Example}[section]
\newtheorem{lemma}{Lemma}[section]

\newtheorem{corollary}{Corollary}[section]
\numberwithin {equation}{section}
\newenvironment{proof}{\textbf{Proof.}} {\hfill $\Box$}

\begin{document}
\title{\bf Further results for the Dunkl Transform and the generalized Ces\`{a}ro operator}
\author{Chokri Abdelkefi and Faten Rached
 \footnote{\small This work was completed with the support of the DGRST research project LR11ES11 and the program CMCU 10G /
1503.}\\ \small  Department of Mathematics, Preparatory
Institute of Engineer Studies of Tunis  \\ \small 1089 Monfleury Tunis, University of Tunis, Tunisia\\
  \small E-mail : chokri.abdelkefi@ipeit.rnu.tn \\  \small E-mail : rached@math.jussieu.fr}%
\date{}
\maketitle
\begin{abstract}
In this paper, we consider Dunkl theory on $\mathbb{R}^d$
 associated to a finite reflection group. This theory generalizes classical Fourier
analysis. First, we give for $1<p\leq2$, sufficient conditions for
weighted $L^p$-estimates of the Dunkl transform of a function $f$
using respectively the modulus of continuity of $f$ in the radial
case and the convolution for $f$ in the general case. In particular,
we obtain as application, the integrability of this transform on
Besov-Lipschitz spaces. Second, we provide necessary and sufficient
conditions on nonnegative functions $\varphi$ defined on $[0,1]$ to
ensure the boundedness of the generalized Ces\`{a}ro operator
$\mathcal{C_\varphi}$ on Herz spaces and we obtain the corresponding
operator norm inequalities.
\end{abstract}
\section{Introduction }
\par Dunkl theory generalizes classical Fourier analysis on
$\mathbb{R}^d$. It started twenty years ago with Dunkl's seminal
work [9] and was further developed by several mathematicians (see
[8, 11, 14, 16, 18]) and later was applied and generalized in
different ways by many authors (see [3, 4, 5, 17]). A key tool in
the study of special functions with reflection symmetries are Dunkl
operators. These are commuting differential-difference operators
$T_i, 1 \leq i \leq d$ associated to an arbitrary finite reflection
group $W$ on an Euclidean space and to a non negative multiplicity
function $k$. The Dunkl kernel $E_k$ has been introduced by C.F.
Dunkl in [10]. This kernel is used to define the Dunkl transform
$\mathcal{F}_k$. K. Trim\`eche has introduced in [18] the Dunkl
translation operators $\tau_x$,
 $x\in \mathbb{R}^d$, on the space of infinitely
differentiable functions on $\mathbb{R}^d$.  At the moment an
explicit formula for the Dunkl translation $\tau_x(f)$ of a function
$f$ is unknown in general. However, such formula is known when the
function $f$ is radial (see next section). In particular, the
boundedness of $\tau_x$ is established in this case. As a result one
obtains a formula for the convolution $\ast_k$. In the case
$k\equiv0$, the $T_i$ reduce to the corresponding partial
derivatives $\frac{\partial}{\partial x_i}$. Therefore Dunkl
analysis can be viewed as a generalization of classical Fourier
analysis (see next section, Remark 2.1). An important motivation to
study Dunkl operators originates in their relevance for the analysis
of quantum many body systems of Calogero-Moser-Sutherland type.
These describe algebraically integrable systems in one dimension and
have gained considerable interest in mathematical
physics (see [19]).\\

Let f be a function in $L^p_k(\mathbb{R}^d)$, $1\leq p<\infty$ where
$L^p_k(\mathbb{R}^d)$ denote the space $L^{p}(\mathbb{R}^d,
w_k(x)dx)$ with $w_k$ the weight function associated to the Dunkl
operators given by
$$w_k(x) = \prod_{\xi\in R_+} |\langle \xi,x\rangle|^{2k(\xi)},
\quad x \in \mathbb{R}^d \quad(\mbox{see next section}).$$ The
modulus of continuity $\omega_{p,k}(f)$ of first order of a radial
function $f$ in $L^p_k(\mathbb{R}^d)$ is defined by
$$\omega_{p,k}(f)(x)=\displaystyle\sup_{0<t\leq
x}\int_{S^{d-1}}\|\tau_{tu}(f)-f\|_{p,k}d\sigma(u),\;\;x>0,$$ where
 $S^{d-1}$ is the unit sphere on $\mathbb{R}^d$ with
the normalized surface measure $d\sigma$.\\We use $\|\ \;\|_{p,k}$
as a shorthand for $\|\ \;\|_{L^p_k( \mathbb{R}^d)}$.\\We set for
$f\in L^p_k(\mathbb{R}^d)$,
\begin{eqnarray*}  \tilde{\omega}_{p,k}(f)(x)=\sup_{0<t\leq
x} \|f \ast_k \phi_t\|_{p,k},\;\;x>0,\end{eqnarray*}where $\phi$ is
a radial function in $\mathcal{S}(\mathbb{R}^d)$ satisfying
\begin{eqnarray*} \exists\,\;c>0\,; \;\,|\mathcal{F}_k(\phi)(y)| >
c\,,\;\,\forall\,\;y\in\Big\{z\in \mathbb{R}^d\,:\;\frac{1}{2}<
\|z\|<1\Big\}.\end{eqnarray*} $\phi_t$ being the dilation of $\phi$
given by
$\phi_t(y)=\frac{1}{t^{2(\gamma+\frac{d}{2})}}\phi(\frac{y}{t})$,
 for all $t\in (0,+\infty)$ and $y\in\mathbb{R}^d$. Note that
 $\mathcal{F}_k(\phi_t)(y)=\mathcal{F}_k(\phi)(ty)$.

Obviously, $\omega_{p,k}(f)(x)$ and $\tilde{\omega}_{p,k}(f)(x)$ are nondecreasing in $x$.\\

For $ \beta>0$, $1\leq p < + \infty$, we define the weighted
Besov-Lipschitz spaces (see [7] for the classical case) denoted by $
\mathcal{B}\mathcal{D}_{p}^{\beta,k} $ as the subspace of radial
functions f in $L^p_k(\mathbb{R}^d)$ satisfying\begin{eqnarray*}
\displaystyle\sup_{x>0}\frac{\omega_{p,k}(f)(x)}{x^{\beta}}<+\infty.\end{eqnarray*}

Let $\varphi$ a nonnegative function defined on $[0,1]$. For a
measurable complex-valued function $f$ on $\mathbb{R}^d$, we define
the generalized Ces\`{a}ro operator $\mathcal{C_\varphi}$ by
   \begin{eqnarray*}\mathcal{C_\varphi}f(x)=\int_0^1 f(\frac{x}{t})\,t^{-(2\gamma+d)}\varphi (t) dt,\;\;x\in
   \mathbb{R}^d,
  \end{eqnarray*}where $\displaystyle{\gamma=\sum_{\xi \in R_+} k(\xi)}$ with $R_+$ a fixed positive root
  system (see next section). If $k\equiv0$, $\varphi \equiv 1$ and $d=1$, we
obtain the classical Ces\`{a}ro average operator $\mathcal{C}$ given
by $$\mathcal{C}f(x) =
\left\{\begin{array}{lll}\int_x^{+\infty} \frac{f(y)}{y} dy, \;&&\mbox{if}\; x >0\,\\&&\\
-\int_{-\infty}^x \frac{f(y)}{y} dy,\; &&\mbox{if}\;
x<0.\end{array}\right.$$

We define for $\beta>0$ and $1 \leq p,q<+\infty $, the homogeneous
Herz-type space for the Dunkl operators $H_{p,q}^{\beta,k}$ (see
[12] for the classical case) by the space of functions $f $ in
$L^p_k(\mathbb{R}^d)_{ \textbf{loc}}$
 satisfying $$   \Big(\sum_{j=-\infty}^{+\infty}(2^{j\beta }\|  f \chi_j \|_{p,k})^q \Big)^{\frac{1}{q}}<+\infty,$$
 where $\chi_j$ is the characteristic function of the set $$ A_j=\{x\in
\mathbb{R}^d\,;\,2^{j-1}\leq \|x\|\leq 2^{j}\}\;\;\mbox{for}\;\;
 j\in\mathbb{Z},$$ and
$L^p_k(\mathbb{R}^d)_{\textbf{loc}}$ is the space $L^{p}_{
\textbf{loc}}(\mathbb{R}^d, w_k(x)dx).$ \\

In this paper, first we give for $1<p\leq2$, sufficient conditions
for weighted $L^p$-estimates of the Dunkl transform
$\mathcal{F}_k(f)$ of a function $f$ using respectively the modulus
of continuity $\omega_{p,k}(f)$ of $f$ in the radial case and the
convolution $\tilde{\omega}_{p,k}(f)$ for $f$ in the general case.
In particular, we obtain as application, the integrability of this
transform on the Besov-Lipschitz spaces
$\mathcal{B}\mathcal{D}_{p}^{\beta,k}$. This is an extension to the
Dunkl transform on higher dimension of the results obtained by F.
M\'{o}ricz in [13] for the Fourier transform on the real line.
Second, we provide necessary and sufficient conditions for
$\mathcal{C_\varphi}$ to be bounded on the Herz spaces $
H_{p,q}^{\beta,k}$ when $\beta>0$, $1< p <+\infty$, $1\leq
q<+\infty$ and we obtain the corresponding
operator norm inequalities.\\

The contents of this paper are as follows: \\In section 2, we
collect some basic definitions and results about harmonic analysis
associated with Dunkl operators.\\
In section 3, we give sufficient conditions for weighted
$L^p$-estimates of the Dunkl transform of function which we apply on
the Besov-Lipschitz spaces.\\In section 4, we provide necessary and
sufficient conditions for
$\mathcal{C_\varphi}$ to be bounded on Herz spaces.\\

Along this paper we denote by $\langle .,.\rangle$ the usual
Euclidean inner product in $\mathbb{R}^d$ as well as its extension
to $\mathbb{C}^d \times \mathbb{C}^d$, we write for $x \in
\mathbb{R}^d, \|x\| = \sqrt{\langle x,x\rangle}$ and we use $c$ to
denote a suitable positive constant which is not necessarily the
same in each occurrence. Furthermore, we denote by

$\bullet\quad \mathcal{E}(\mathbb{R}^d)$ the space of infinitely
differentiable functions on $\mathbb{R}^d$.

$\bullet\quad \mathcal{S}(\mathbb{R}^d)$ the Schwartz space of
functions in $\mathcal{E}( \mathbb{R}^d)$ which are rapidly
decreasing as well as their derivatives.

$\bullet\quad \mathcal{D}(\mathbb{R}^d)$ the subspace of
$\mathcal{E}(\mathbb{R}^d)$ of compactly supported functions.
\section{Preliminaries}
 $ $ In this section, we recall some notations and results in Dunkl
theory and we refer for more details to the surveys [15].\\

Let $W$ be a finite reflection group on $\mathbb{R}^{d}$, associated
with a root system $R$. For $\alpha\in R$, we denote by
$\mathbb{H}_\alpha$ the hyperplane orthogonal to $\alpha$. Given
$\beta\in\mathbb{R}^d\backslash\bigcup_{\alpha\in R}
\mathbb{H}_\alpha$, we fix a positive subsystem $R_+=\{\alpha\in R:
\langle \alpha,\beta\rangle>0\}$. We denote by $k$ a nonnegative
multiplicity function defined on $R$ with the property that $k$ is
$W$-invariant. We associate with $k$ the index
$$\gamma = \sum_{\xi \in R_+} k(\xi) \geq 0,$$
and and the weight function $w_k$ given by
\begin{eqnarray*} w_k(x) = \prod_{\xi\in R_+} |\langle
\xi,x\rangle|^{2k(\xi)}, \quad x \in \mathbb{R}^d.\end{eqnarray*}
$w_k$ is $W$-invariant and homogeneous of degree $2\gamma$.\\

Further, we introduce the Mehta-type constant $c_k$ by
$$c_k = \left(\int_{\mathbb{R}^d} e^{- \frac{\|x\|^2}{2}}
w_k (x)dx\right)^{-1}.$$

For every $1 \leq p \leq + \infty$, we denote by
$L^p_k(\mathbb{R}^d)$, the space $L^{p}(\mathbb{R}^d, w_k (x)dx)$
and we use $\|\ \;\|_{p,k}$
as a shorthand for $\|\ \;\|_{L^p_k( \mathbb{R}^d)}$. \\

By using the homogeneity of degree $2\gamma$ of $w_k$, it is shown
in [14] that for a radial function $f$ in $L^1_k ( \mathbb{R}^d)$,
there exists a function $F$ on $[0, + \infty)$ such that $f(x) =
F(\|x\|)$, for all $x \in \mathbb{R}^d$. The function $F$ is
integrable with respect to the measure $r^{2\gamma+d-1}dr$ on $[0, +
\infty)$ and we have
 \begin{eqnarray} \int_{\mathbb{R}^d}  f(x)\,d\nu_k(x)&=&\int^{+\infty}_0
\Big( \int_{S^{d-1}}f(ry)w_k(ry)d\sigma(y)\Big)r^{d-1}dr\nonumber\\
&=&
 \int^{+\infty}_0
\Big( \int_{S^{d-1}}w_k(ry)d\sigma(y)\Big)
F(r)r^{d-1}dr\nonumber\\&= & d_k\int^{+ \infty}_0 F(r)
r^{2\gamma+d-1}dr,
\end{eqnarray}
  where $S^{d-1}$
is the unit sphere on $\mathbb{R}^d$ with the normalized surface
measure $d\sigma$  and \begin{eqnarray}d_k=\int_{S^{d-1}}w_k
(x)d\sigma(x) = \frac{c^{-1}_k}{2^{\gamma +\frac{d}{2} -1}
\Gamma(\gamma + \frac{d}{2})}\;.  \end{eqnarray}

The Dunkl operators $T_j\,,\ \ 1\leq j\leq d\,$, on $\mathbb{R}^d$
associated with the reflection group $W$ and the multiplicity
function $k$ are the first-order differential- difference operators
given by
$$T_jf(x)=\frac{\partial f}{\partial x_j}(x)+\sum_{\alpha\in R_+}k(\alpha)
\alpha_j\,\frac{f(x)-f(\rho_\alpha(x))}{\langle\alpha,x\rangle}\,,\quad
f\in\mathcal{E}(\mathbb{R}^d)\,,\quad x\in\mathbb{R}^d\,,$$ where
$\rho_\alpha$ is the reflection on the hyperplane
$\mathbb{H}_\alpha$ and $\alpha_j=\langle\alpha,e_j\rangle,$
$(e_1,\ldots,e_d)$ being the canonical basis of $\mathbb{R}^d$.
\begin{remark}In the case $k\equiv0$, the weighted function $w_k\equiv1$ and the measure associated to the
Dunkl operators coincide with the Lebesgue measure. The $T_j$ reduce
to the corresponding partial derivatives. Therefore Dunkl analysis
can be viewed as a generalization of classical Fourier
analysis.\end{remark}

For $y \in \mathbb{C}^d$, the system
$$\left\{\begin{array}{lll}T_ju(x,y)&=&y_j\,u(x,y),\qquad1\leq j\leq d\,,\\  &&\\
u(0,y)&=&1\,.\end{array}\right.$$ admits a unique analytic solution
on $\mathbb{R}^d$, denoted by $E_k(x,y)$ and called the Dunkl
kernel. This kernel has a unique holomorphic extension to
$\mathbb{C}^d \times \mathbb{C}^d $. We have for all $\lambda\in
\mathbb{C}$ and $z, z'\in \mathbb{C}^d,\;
 E_k(z,z') = E_k(z',z)$,  $E_k(\lambda z,z') = E_k(z,\lambda z')$ and for $x, y
\in \mathbb{R}^d,\;|E_k(x,iy)| \leq 1$.\\

The Dunkl transform $\mathcal{F}_k$ is defined for $f \in
\mathcal{D}( \mathbb{R}^d)$ by
$$\mathcal{F}_k(f)(x) =c_k\int_{\mathbb{R}^d}f(y) E_k(-ix, y)w_k(y)dy,\quad
x \in \mathbb{R}^d.$$  We list some known properties of this
transform:
\begin{itemize}
\item[i)] The Dunkl transform of a function $f
\in L^1_k( \mathbb{R}^d)$ has the following basic property
\begin{eqnarray*}\| \mathcal{F}_k(f)\|_{\infty,k} \leq
 \|f\|_{ 1,k}\;. \end{eqnarray*}
\item[ii)] The Dunkl transform is an automorphism on the Schwartz space $\mathcal{S}(\mathbb{R}^d)$.
\item[iii)] When both $f$ and $\mathcal{F}_k(f)$ are in $L^1_k( \mathbb{R}^d)$,
 we have the inversion formula \begin{eqnarray*} f(x) =   \int_{\mathbb{R}^d}\mathcal{F}_k(f)(y) E_k( ix, y)w_k(y)dy,\quad
x \in \mathbb{R}^d.\end{eqnarray*}
\item[iv)] (Plancherel's theorem) The Dunkl transform on $\mathcal{S}(\mathbb{R}^d)$
 extends uniquely to an isometric automorphism on
$L^2_k(\mathbb{R}^d)$.
\end{itemize} Since the Dunkl transform $\mathcal{F}_k(f)$ is of strong-type $(1,\infty)$ and
$(2,2)$, then by interpolation, we get for $f \in
L^p_k(\mathbb{R}^d)$ with $1\leq p\leq 2$ and $p'$ such that
$\frac{1}{p}+\frac{1}{p'}=1$, the Hausdorff-Young inequality
\begin{eqnarray*}
\|\mathcal{F}_k(f)\|_{p',k}\leq c\,\|f\|_{p,k}.
\end{eqnarray*}
The Dunkl transform of a radial function in $L^1_k( \mathbb{R}^d)$
 is also radial and could be expressed via the Hankel transform. More precisely, according to ([14], proposition
2.4), we have the following results:
\begin{eqnarray}\int_{S^{d-1}}E_k(ix,y)w_k(y)d\sigma(y) =
d_k\, j_{\gamma + \frac{d}{2}-1}(\|x\|), \end{eqnarray}
 and
 \begin{eqnarray}\mathcal{F}_k(f)(x) &=&\int^{+\infty}_0
\Big( \int_{S^{d-1}}E_k(-ix, y)w_k(y)d\sigma(y)\Big)
F(r)r^{2\gamma+d-1}dr\nonumber\\&=&  c_k^{-1}\; \mathcal{H}_{\gamma
+ \frac{d}{2}-1} (F)(\|x\|),\quad x \in \mathbb{R}^d, \end{eqnarray}
where $F$ is the function defined on $[ 0, + \infty)$ by $F(\|x\|) =
f(x),\; x \in \mathbb{R}^d$, $ \mathcal{H}_{\gamma + \frac{d}{2}-1}$
is the Hankel transform of order $\gamma + \frac{d}{2}-1$ and
$j_{\gamma + \frac{d}{2}-1}$  the normalized Bessel function of the
first kind and order $\gamma + \frac{d}{2}-1$.

The integral representation of $j_{\alpha}$, $\alpha>-\frac{1}{2}$
is given by
\begin{eqnarray}
j_{\alpha}(\lambda x)=\frac{2\Gamma(\alpha
+1)}{\sqrt{\pi}\Gamma(\alpha+\frac{1}{2})}\int^{1}_{0}(1-t^{2})^{\alpha-\frac{1}{2}}\cos(\lambda
x t)dt\quad \text{for}\, x\in\mathbb{R}\;\text{and}\;
\lambda\in\mathbb{C}.
\end{eqnarray}
We note that
\begin{eqnarray}
\frac{2\Gamma(\alpha
+1)}{\sqrt{\pi}\Gamma(\alpha+\frac{1}{2})}\int^{1}_{0}(1-t^{2})^{\alpha-\frac{1}{2}}dt=1.
\end{eqnarray} \\

K. Trim\`eche has introduced in [18] the Dunkl translation operators
$\tau_x$, $x\in\mathbb{R}^d$, on $\mathcal{E}( \mathbb{R}^d)\;.$ For
$f\in \mathcal{S}( \mathbb{R}^d)$ and $x, y\in\mathbb{R}^d$, we have
\begin{eqnarray}\mathcal{F}_k(\tau_x(f))(y)=E_k(i x, y)\mathcal{F}_k(f)(y).\end{eqnarray}  Notice that for all $x,y\in\mathbb{R}^d$,
$\tau_x(f)(y)=\tau_y(f)(x)$ and for fixed $x\in\mathbb{R}^d$
\begin{eqnarray*}\tau_x \mbox{\; is a continuous linear mapping from \;}
\mathcal{E}( \mathbb{R}^d) \mbox{\;
into\;}\mathcal{E}(\mathbb{R}^d)\,.\end{eqnarray*} As an operator on
$L_k^2(\mathbb{R}^d)$, $\tau_x$ is bounded. A priori it is not at
all clear whether the translation operator can be defined for $L^p$-
functions with $p$ different from 2. However, according to ([17],
Theorem 3.7), the operator $\tau_x$ can be extended to the space of
radial functions in $L^p_k(\mathbb{R}^d),$ $1 \leq p \leq 2$ and we
have for a radial function $f \in L^p_k(\mathbb{R}^d)$
\begin{eqnarray}\|\tau_x(f)\|_{p,k} \leq \|f\|_{p,k}.\end{eqnarray}

The Dunkl convolution product $\ast_k$ of two functions $f$ and $g$
in $L^2_k(\mathbb{R}^d)$ is given by
$$(f\; \ast_k g)(x) = \int_{\mathbb{R}^d} \tau_x (f)(-y) g(y) w_k(y)dy,\quad
x \in \mathbb{R}^d .$$ The Dunkl convolution product is commutative
and for $f,\,g \in \mathcal{D}( \mathbb{R}^d)$, we have
\begin{eqnarray}\mathcal{F}_k(f\,\ast_k\, g) =
\mathcal{F}_k(f) \mathcal{F}_k(g).\end{eqnarray}It was shown in
([17], Theorem 4.1) that when $g$ is a bounded radial function in
$L^1_k( \mathbb{R}^d)$, then
\begin{eqnarray*}(f\; \ast_k g)(x) = \int_{\mathbb{R}^d}  f(y) \tau_x (g)(-y) w_k(y)dy,\quad
x \in \mathbb{R}^d , \end{eqnarray*} initially defined on the
intersection of $L^1_k(\mathbb{R}^d)$ and $L^2_k(\mathbb{R}^d)$
extends to $L^p_k(\mathbb{R}^d)$, $1\leq p\leq +\infty$ as a bounded
operator. In particular, \begin{eqnarray}\|f \ast_k g\|_{p,k} \leq
\|f\|_{p,k} \|g\|_{1,k}\,.\end{eqnarray}
\section{Weighted $L^p$-estimates for the Dunkl transform with sufficient conditions}
In this section, we give sufficient conditions for weighted
$L^p$-estimates of the Dunkl transform of function which we apply on
the Besov-Lipschitz space.\\

Throughout this section, we denote by $p'$ the conjugate
$\frac{p}{p-1}$ of $p$ for $1<p\leq 2$. According to (2.8) and
(2.10), we recall that :\\$\bullet$ The modulus of continuity
$\omega_{p,k}(f)$ of first order of a radial function $f$ in
$L^p_k(\mathbb{R}^d)$ is defined by
$$\omega_{p,k}(f)(x)=\displaystyle\sup_{0<t\leq
x}\int_{S^{d-1}}\|\tau_{tu}(f)-f\|_{p,k}d\sigma(u),\;\;x>0.$$ \\
$\bullet$ We set for $f\in L^p_k(\mathbb{R}^d)$,
\begin{eqnarray*}  \tilde{\omega}_{p,k}(f)(x)=\sup_{0<t\leq
x} \|f \ast_k \phi_t\|_{p,k},\;\;x>0,\end{eqnarray*}where $\phi$ is
a radial function in $\mathcal{S}(\mathbb{R}^d)$ satisfying
\begin{eqnarray*} \exists\,\;c>0\,; \;\,|\mathcal{F}_k(\phi)(y)| >
c\,,\;\,\forall\,\;y\in\Big\{z\in \mathbb{R}^d\,:\;\frac{1}{2}<
\|z\|<1\Big\}.\end{eqnarray*} $\phi_t$ being the dilation of $\phi$
given by
$\phi_t(y)=\frac{1}{t^{2(\gamma+\frac{d}{2})}}\phi(\frac{y}{t})$,
 for all $t\in (0,+\infty)$ and $y\in\mathbb{R}^d$.\\

For $\theta\geq 1$, we introduce a class of nonnegative
$w_k$-locally integrable radial functions on
$\displaystyle\{x\in\mathbb{R}^d: \|x\|\geq1\}$ which we denote by
$G_\theta$. We said that a function $g$ belongs to the class
$G_\theta$ if there exists a constant $\kappa_{\theta}\geq 1$ such
that for $\eta=1,2,...$
\begin{eqnarray}
\left(\int_{1\leq\|z\|\leq2}g(2^\eta z)^{\theta}
w_{k}(z)dz\right)^{\frac{1}{\theta}}\leq \kappa_{\theta}
\int_{\frac{1}{2}\leq\|z\|\leq1}g(2^\eta z)  w_{k}(z)dz.
\end{eqnarray}
 If we set
\begin{eqnarray*}
C_{\eta}=\displaystyle\{y\in\mathbb{R}^d:
2^{\eta}\leq\|y\|<2^{\eta+1}\}, \quad\mbox{ for}\; \eta=0,1,2,...,
\end{eqnarray*} then using (2.1) and the change of variables $y=2^\eta
z$, we can write (3.1) in the form
\begin{eqnarray}
\left(\int_{C_{\eta}}g(y)^{\theta}
w_{k}(y)dy\right)^{\frac{1}{\theta}}\leq
\kappa_{\theta}\;2^{\eta(\frac{1-\theta}{\theta})(2\gamma+d)}\int_{C_{\eta-1}}g(
y) w_{k}(y)dy.
\end{eqnarray}
\begin{example} If $g\equiv const$, then $g\in G_\theta$ where we choose $$\displaystyle\kappa_{\theta}\geq
2^{2\gamma+d}\Big(\frac{2^{2\gamma+d}-1}{2\gamma+d}
\Big)^{\frac{1}{\theta}-1}.$$
\end{example}
\begin{theorem}
Let $1<p\leq 2$ and $f$ be a radial function in $
L^{p}_{k}(\mathbb{R}^{d})$. Then for $1\leq q\leq p^{\prime}$ and
$g\in G_{\frac{p}{p-qp+q}}$, we have
\begin{eqnarray*}
\int_{\|y\|\geq 2}g(y)|\mathcal{F}_k(f)(y)|^{q}w_{k}(y)dy\leq C
\int_{\|y\|\geq1}g(y)\|y\|^{-\frac{q}{p^{\prime}}(2\gamma+d)}\omega_{p,k}^{q}(f)(\frac{\pi}{\|y\|})w_{k}(y)dy,
\end{eqnarray*}
where $C$ is a constant depending only on $p$ and $q$.
\end{theorem}
\begin{proof}
Let $f$ be a radial function in $ L^{p}_{k}(\mathbb{R}^{d})$ for
$1<p\leq 2$, then by (2.7) we have
\begin{eqnarray*}
\mathcal{F}_{k}(\tau_{xu}(f)-f)(y)=(E_{k}(ixu,y)-1)\mathcal{F}_{k}f(y),
\end{eqnarray*}
$u\in S^{d-1}$, $x\in [0,+\infty)$ and a.e $y\in \mathbb{R}^{d}$.\\
We can assert by Haussdorf-Young's inequality that
\begin{eqnarray}
\|\mathcal{F}_{k}(\tau_{xu}(f)-f)\|_{p^{\prime},k}&=&\left(\int_{\mathbb{R}^{d}}|\mathcal{F}_{k}f(y)|^{p^{\prime}}
|E_{k}(ixu,y)-1|^{p^{\prime}}w_{k}(y)dy\right)^{\frac{1}{p^{\prime}}}\nonumber\\
&\leq& c \,\|\tau_{xu}(f)-f\|_{p,k}.
\end{eqnarray}
On the other hand, from (2.1) and (2.4), we get\\
$c_{k}\displaystyle\int_{\mathbb{R}^{d}}|\mathcal{F}_{k}f(y)|^{p^{\prime}}|E_{k}(ixu,y)-1|^{p^{\prime}}w_{k}(y)dy$
\begin{eqnarray*}
=\int^{+\infty}_{0}|\mathcal{H}_{\gamma+\frac{d}{2}-1}(F)(r)|
\left(\int_{S^{d-1}}|E_{k}(irxu,z)-1|^{p^{\prime}}w_{k}(z)d\sigma(z)\right)r^{2\gamma+d-1}dr,\qquad
\end{eqnarray*}
where $r=\|y\|$ and $F$ is the function defined on $(0,+\infty)$ by
$F(\|y\|)=f(y)$ for all $y\in \mathbb{R}^{d}$.\\
By (2.2), (2.3) and H\"older's inequality, we have\\$d_k |j_{\gamma
+ \frac{d}{2}-1}(rx)-1|$
 \begin{eqnarray*}&=& |\int_{S^{d-1}}[E_k(i
xru,z)-1]w_k(z)d\sigma(z)|\nonumber\\
 &\leq&\Big(\int_{S^{d-1}}w_k(z)d\sigma(z)\Big)^{\frac{1}{p}}
\Big(\int_{S^{d-1}}|E_k(i xru,z)-1|^{p'} w_k(z)d\sigma(z)\Big)^{\frac{1}{p'}}\nonumber\\
 &\leq& d_k^{\frac{1}{p}} \Big(\int_{S^{d-1}}|E_k(i xru,z)-1|^{p'}
 w_k(z)d\sigma(z)\Big)^{\frac{1}{p'}}.
\end{eqnarray*}
According to (3.3), it follows that
\begin{eqnarray}
 \int^{+\infty}_{0}|\mathcal{H}_{\gamma+\frac{d}{2}-1}(F)(r)|^{p^{\prime}}|j_{\gamma+\frac{d}{2}-1}(rx)-1|^{p^{\prime}}r^{2\gamma+d-1}dr
\leq c\,\|\tau_{xu}(f)-f\|_{p,k}^{p'}.
\end{eqnarray}
Integrating the two members of (3.4) over $S^{d-1}$, this yields
\begin{eqnarray}
 \left(\int^{+\infty}_{0}|\mathcal{H}_{\gamma+\frac{d}{2}-1}(F)(r)|^{p^{\prime}}
|j_{\gamma+\frac{d}{2}-1}(rx)-1|^{p^{\prime}}r^{2\gamma+d-1}dr\right)^{\frac{1}{p^{\prime}}}
\leq c \,\omega_{p,k}(f)(x).
\end{eqnarray}
From (2.5) and (2.6), we get
\begin{eqnarray*}
|j_{\gamma+\frac{d}{2}-1}(rx)-1|=
\frac{4\Gamma(\gamma+\frac{d}{2})}{\sqrt{\pi}\Gamma(\gamma+\frac{d}{2}-\frac{1}{2})}\int^{1}_{0}(1-t^{2})^{\gamma+\frac{d}{2}-\frac{3}{2}}\sin^{2}(\frac{rxt}{2})dt,
\end{eqnarray*}
then, we obtain
\begin{eqnarray*}
\frac{4\Gamma(\gamma+\frac{d}{2})}{\sqrt{\pi}\Gamma(\gamma+\frac{d}{2}-\frac{1}{2})}
\int_{\frac{1}{2}}^{1}(1-t^{2})^{\gamma+\frac{d}{2}-\frac{3}{2}}\sin^{2}(\frac{rxt}{2})dt\leq
|j_{\gamma+\frac{d}{2}-1}(rx)-1|.
\end{eqnarray*}
Let $x=\frac{\pi}{2^{\eta+1}}$, $\eta=1,2,...$ For
$t\in[\frac{1}{2},1[$ and $y\in C_{\eta}$, we have
$$\frac{\pi}{8}\leq\frac{rxt}{2}\leq\frac{\pi}{2},$$ which gives
that
\begin{eqnarray*}
c=\frac{4\Gamma(\gamma+\frac{d}{2})\sin^{2}(\frac{\pi}{8})}{\sqrt{\pi}\Gamma(\gamma+\frac{d}{2}-\frac{1}{2})}
\int_{\frac{1}{2}}^{1}(1-t^{2})^{\gamma+\frac{d}{2}-\frac{3}{2}}dt\leq
|j_{\gamma+\frac{d}{2}-1}(r\frac{\pi}{2^{\eta+1}})-1|.
\end{eqnarray*}
Hence from (3.5), we find that
\begin{eqnarray}
\left(\int_{2^\eta}^{2^{\eta+1}}|\mathcal{H}_{\gamma+\frac{d}{2}-1}(F)(r)|^{p^{\prime}}r^{2\gamma+d-1}dr\right)^{\frac{1}{p^{\prime}}}\leq
c\, \omega_{p,k}(f)(\frac{\pi}{2^{\eta+1}}).
\end{eqnarray}
Now, take $g\in G_{\frac{p}{p-qp+q}}$ and put
$\tilde{g}(\|y\|)=g(y)$, for $y\in \mathbb{R}^{d}$. Applying
H\"{o}lder's inequality,
it follows from (2.1), (3.2) and (3.6) that\\
$
\displaystyle\int_{2^\eta}^{2^{\eta+1}}\tilde{g}(r)|\mathcal{H}_{\gamma+\frac{d}{2}-1}(F)(r)|^{q}r^{2\gamma+d-1}dr
$
\begin{eqnarray*}
&\leq&\left(\int_{2^\eta}^{2^{\eta+1}}|\mathcal{H}_{\gamma+\frac{d}{2}-1}(F)(r)|^{p^{\prime}}r^{2\gamma+d-1}dr\right)^{\frac{q}{p^{\prime}}}
\left(\int_{2^\eta}^{2^{\eta+1}}\tilde{g}^{\frac{p^{\prime}}{p^{\prime}-q}}(r)r^{2\gamma+d-1}dr\right)^{\frac{p^{\prime}-q}{p^{\prime}}}
\\&\leq& c\;
\kappa_{\frac{p^{\prime}}{p^{\prime}-q}}2^{-\eta\frac{q}{p^{\prime}}(2\gamma+d)}\omega^{q}_{p,k}(f)(\frac{\pi}{2^{\eta+1}})\int_{2^{\eta-1}}^{2^{\eta}}\tilde{g}(r)r^{2\gamma+d-1}dr.
\end{eqnarray*}
Then, we deduce
 \\$\displaystyle\int^{+\infty}_{2}\tilde{g}(r)|\mathcal{H}_{\gamma+\frac{d}{2}-1}(F)(r)|^{q}r^{2\gamma+d-1}dr$
\begin{eqnarray*}
\leq c\;
\kappa_{\frac{p^{\prime}}{p^{\prime}-q}}\displaystyle\sum^{+\infty}_{\eta=1}
2^{-\eta\frac{q}{p^{\prime}}(2\gamma+d)}\omega^{q}_{p,k}(f)(\frac{\pi}{2^{\eta+1}})
\int_{2^{\eta-1}}^{2^{\eta}}\tilde{g}(r)r^{2\gamma+d-1}dr.
\end{eqnarray*}
Using the monotonicity of $\omega_{p,k}(f)(x)$ in $x>0$ and the fact
that $\frac{\pi}{2^{\eta+1}}\leq \frac{\pi}{r}$,
$2^{-\eta\frac{q}{p^{\prime}}(2\gamma+d)}\leq
r^{-\frac{q}{p^{\prime}}(2\gamma+d)}$, we obtain\\
$\displaystyle\int^{+\infty}_{2}\tilde{g}(r)|\mathcal{H}_{\gamma+\frac{d}{2}-1}(F)(r)|^{q}r^{2\gamma+d-1}dr$
\begin{eqnarray}&\leq&
c\;
\kappa_{\frac{p^{\prime}}{p^{\prime}-q}}\displaystyle\sum^{+\infty}_{\eta=1}\int_{2^{\eta-1}}^{2^{\eta}}\tilde{g}(r)r^{-\frac{q}{p^{\prime}}(2\gamma+d)}\omega^{q}_{p,k}(f)(\frac{\pi}{r})r^{2\gamma+d-1}dr\nonumber
\\
&\leq&c\;
\kappa_{\frac{p^{\prime}}{p^{\prime}-q}}\int^{+\infty}_{1}\tilde{g}(r)r^{-\frac{q}{p^{\prime}}(2\gamma+d)}\omega^{q}_{p,k}(f)(\frac{\pi}{r})r^{2\gamma+d-1}dr.
\end{eqnarray}
By (2.1), (2.4) and (3.7), we conclude that
\begin{eqnarray*}
\int_{\|y\|\geq 2}g(y)|\mathcal{F}_{k}(f)(y)|^{q}w_{k}(y)dy\leq C
\int_{\|y\|\geq1}g(y)\|y\|^{-\frac{q}{p^{\prime}}(2\gamma+d)}\omega_{p,k}^{q}(f)(\frac{\pi}{\|y\|})w_{k}(y)dy,
\end{eqnarray*}
where $C=c\; \kappa_{\frac{p^{\prime}}{p^{\prime}-q}}$ is a constant
depending only on $p$ and $q$. Our theorem is proved.
\end{proof}

In particular, for the case $g\equiv const$ and $f \in
\mathcal{B}\mathcal{D}_{p}^{\beta,k}$, we obtain as application of
Theorem 3.1, the following result which is of special interest and
was proved in [1, Theorem 3.2].
\begin{corollary}
Let $\beta > 0$, $1 < p \leq 2$ and $f \in
\mathcal{B}\mathcal{D}_{p}^{\beta,k} $, then
\begin{enumerate}
\item For $0 < \beta \leq \frac{2(\gamma+\frac{d}{2})}{p}$, we
have
$$\mathcal{F}_k(f) \in L^{q}_{k}(\mathbb{R}^d)\mbox{ provided that
} \frac{2(\gamma +\frac{d}{2})p}{\beta p+2(\gamma+\frac{d}{2})(p-1)}
< q \leq p'.$$
\item For $\beta > \frac{2(\gamma+\frac{d}{2})}{p}$,
we have
 $\mathcal{F}_k(f) \in L^{1}_{k}(\mathbb{R}^d)$.
\end{enumerate}
\end{corollary}
\begin{theorem}
Let $1<p\leq 2$ and $f$ be a function in $
L^{p}_{k}(\mathbb{R}^{d})$. Then for $1\leq q\leq p^{\prime}$ and
$g\in G_{\frac{p}{p-qp+q}}$, we have
\begin{eqnarray*}
\int_{\|y\|\geq 2}g(y)|\mathcal{F}_k(f)(y)|^{q}w_{k}(y)dy\leq C
\int_{\|y\|\geq1}g(y)\|y\|^{-\frac{q}{p^{\prime}}(2\gamma+d)}\tilde{\omega}_{p,k}^{q}(f)(\frac{1}{\|y\|})w_{k}(y)dy,
\end{eqnarray*}
where $C$ is a constant depending only on $p$ and $q$.
\end{theorem}
\begin{proof}
Let $f$ be a function in $ L^{p}_{k}(\mathbb{R}^{d})$ for $1<p\leq
2$, then by (2.9) we have for $x \in (0,+\infty)$,
\begin{eqnarray*}
\mathcal{F}_k( f\ast_k \phi_x)(y) &=& \mathcal{F}_k(f)(y)
\mathcal{F}_k(\phi_x)(y)\,,\\&=&\mathcal{F}_k(f)(y)
\mathcal{F}_k(\phi)(xy)\,, \, \mbox{ a.e. } y\in \mathbb{R}^d.
\end{eqnarray*}Using Haussdorf-Young's inequality, we obtain
\begin{eqnarray}
\|\mathcal{F}_{k}(f\ast_k\phi_x)\|_{p^{\prime},k}&=&\left(\int_{\mathbb{R}^{d}}|\mathcal{F}_{k}f(y)|^{p^{\prime}}
|\mathcal{F}_k(\phi)(xy)|^{p^{\prime}}w_{k}(y)dy\right)^{\frac{1}{p^{\prime}}}\nonumber\\
&\leq& c \,\|f\ast_k\phi_x\|_{p,k}\nonumber\\&\leq& c
\,\tilde{\omega}_{p,k}^{q}(f)(x).
\end{eqnarray}Let $x=\frac{1}{2^{\eta+1}}$ and $\eta=1,2,...$ For $y\in C_{\eta}$, we have
$$\frac{1}{2}\leq x\|y\|\leq1,$$ this gives using (3.8) and the property of the function
$\phi$,
\begin{eqnarray}
\left(\int_{C_\eta}|\mathcal{F}_{k}f(y)|^{p^{\prime}}
 w_{k}(y)dy\right)^{\frac{1}{p^{\prime}}} \leq c
\,\tilde{\omega}_{p,k}^{q}(f)(\frac{1}{2^{\eta+1}}).
\end{eqnarray}
Now, take $g\in G_{\frac{p}{p-qp+q}}$. Applying H\"{o}lder's
inequality, it follows from (3.2) and (3.9) that\\
$ \displaystyle\int_{C_\eta}g(y)|\mathcal{F}_{k}f(y)|^{q}w_{k}(y)dy
$
\begin{eqnarray*}
&\leq&\left(\int_{C_\eta}|\mathcal{F}_{k}f(y)|^{p^{\prime}}w_{k}(y)dy\right)^{\frac{q}{p^{\prime}}}
\left(\int_{C_\eta}g^{\frac{p^{\prime}}{p^{\prime}-q}}(y)w_{k}(y)dy\right)^{\frac{p^{\prime}-q}{p^{\prime}}}
\\&\leq& c\;
\kappa_{\frac{p^{\prime}}{p^{\prime}-q}}2^{-\eta\frac{q}{p^{\prime}}(2\gamma+d)}\tilde{\omega}^{q}_{p,k}(f)(\frac{1}{2^{\eta+1}})\int_{C_{\eta-1}}g(y)w_{k}(y)dy.
\end{eqnarray*}
Then, we deduce
 \\$ \displaystyle\int_{\|y\|\geq2}g(y)|\mathcal{F}_{k}f(y)|^{q}w_{k}(y)dy
$
\begin{eqnarray*}
\leq c\;
\kappa_{\frac{p^{\prime}}{p^{\prime}-q}}\displaystyle\sum^{+\infty}_{\eta=1}2^{-\eta\frac{q}{p^{\prime}}(2\gamma+d)}\tilde{\omega}^{q}_{p,k}(f)(\frac{1}{2^{\eta+1}})\int_{C_{\eta-1}}g(y)w_{k}(y)dy.
\end{eqnarray*}
 As in the proof of Theorem 3.1, using the monotonicity of
$\tilde{\omega}_{p,k}(f)(x)$ in $x>0$, we conclude that
\begin{eqnarray*}
\int_{\|y\|\geq 2}g(y)|\mathcal{F}_{k}(f)(y)|^{q}w_{k}(y)dy\leq C
\int_{\|y\|\geq1}g(y)\|y\|^{-\frac{q}{p^{\prime}}(2\gamma+d)}\tilde{\omega}_{p,k}^{q}(f)(\frac{1}{\|y\|})w_{k}(y)dy,
\end{eqnarray*}
where $C=c\; \kappa_{\frac{p^{\prime}}{p^{\prime}-q}}$ is a constant
depending only on $p$ and $q$. This completes the proof.
\end{proof}
\begin{remark} As consequence, from Theorem 3.2, we deduce in the
particular case when $g\equiv const$ and $f \in L^p_k(
\mathbb{R}^d)$ satisfying \begin{eqnarray*}
\displaystyle\sup_{x>0}\frac{\tilde{\omega}_{p,k}(f)(x)}{x^{\beta}}<+\infty,\end{eqnarray*}
similar results to those obtained in corollary 3.1.
\end{remark}
\begin{remark} It was shown in [2, Theorem 3.1] that the Dunkl transform is a continuous map
between a class of Besov spaces and Herz spaces.
\end{remark}
\section{Boundedness of the Ces\`{a}ro operator on Herz spaces}
\label{sec:3} In this section, we provide necessary and sufficient
conditions for the generalized Ces\`{a}ro operator
$\mathcal{C_\varphi}$ to be bounded on Herz spaces. Before, we start
with some useful lemmas.
\begin{lemma}
Let $1\leq q<\infty$. If $f$ be a non-negative measurable function
on $[0,1]$, then
 \begin{eqnarray}
\Big(\int_{0}^{1}f\Big)^{q}\leq \int_{0}^{1}f^q.
\end{eqnarray}
\end{lemma}
\begin{lemma} (see E. F. Beckenbach and R. Bellman [6])
Let $1\leq q<\infty$. If $f$ be a non-negative measurable and
concave function on $[0,1]$, then
\begin{eqnarray}
\Big(\int_{0}^{1}f\Big)^{\frac{1}{q}}\leq
\frac{1+\frac{1}{q}}{2^\frac{1}{q}}\int_{0}^{1}f^\frac{1}{q}.
\end{eqnarray}
\end{lemma}
\begin{theorem}
Let $\beta>0$, $1< p <+\infty$, $1\leq q<+\infty$ and $\varphi$ be a
real-valued non-negative measurable function defined on $[0,1]$.
Suppose that $t\mapsto t^{-(2\gamma+d)(1-\frac{1}{p})}\varphi(t)$ is
a concave function on $[0,1]$, then the generalized Ces\`{a}ro
operator $\mathcal{C_\varphi}$ can be extended to a bounded operator
from $H^{\beta,k}_{p,q}$ into itself if and only if
\begin{eqnarray} \int_0^1
t^{\beta-(2\gamma+d)(1-\frac{1}{p})}\varphi(t)dt<+\infty.
\end{eqnarray}
Moreover, when (4.3) holds, the operator norm $\|
\mathcal{C_\varphi}\|$ of $\mathcal{C_{\varphi}}$ on
$H^{\beta,k}_{p,q}$ satisfies the following inequality
\begin{eqnarray*}\int_0^1 t^{\beta-(2\gamma+d)(1-\frac{1}{p})}\varphi(t)dt\leq
\|\mathcal{C_\varphi}\|\leq c_{q,\beta}\int_0^1
t^{\beta-(2\gamma+d)(1-\frac{1}{p})}\varphi(t)dt,
\end{eqnarray*}
with $ c_{q,\beta}=2^{(1-\frac{2}{q})}(1+\frac{1}{q})(1+2^{\beta}).$
 \end{theorem}
\begin{proof} Let $\beta>0$, $1< p <+\infty$, $1\leq q<+\infty$ and $f\in
H^{\beta,k}_{p,q}$.

Suppose that (4.3) holds. Using the Minkowski inequality for
integrals and the homogeneity of degree $2\gamma$ of $w_k$, we get
\begin{eqnarray*}
\|(\mathcal{C_\varphi}f)\chi_j\|_{p,k}&=&\Big(\int_{A_j}\Big|\int_{0}^{1}
f(\frac{x}{t})t^{-(2\gamma+d)}\varphi(t)dt\Big|^p w_k(x)dx)\Big)^\frac{1}{p}\\
&\leq& \int_{0}^{1}\Big(\int_{A_{j}} |f(\frac{x}{t})|^p
w_k(x)dx\Big)^\frac{1}{p}t^{-(2\gamma+d)}\varphi(t)dt\\
&=&\int_{0}^{1}\Big(\int_{\frac{2^{j-1}}{t}\leq\|u\|\leq
\frac{2^{j}}{t}} |f(u)|^p
w_k(u)du\Big)^\frac{1}{p}t^{-(2\gamma+d)(1-\frac{1}{p})}\varphi(t)dt.
\end{eqnarray*}
Put $\psi(t)=t^{-(2\gamma+d)(1-\frac{1}{p})}\varphi(t)$,
$t\in(0,1)$. Since for each $t\in(0,1)$, there exists an integer $m$
such that $2^{m-1}\leq t\leq 2^{m}$, then we obtain
\begin{eqnarray*}
\|(\mathcal{C_\varphi}f)\chi_j\|_{p,k}&\leq &\int_{0}^{1}
(\|f\chi_{j-m}\|_{p,k}^p+
\|f\chi_{j-m+1}\|_{p,k}^p )^{\frac{1}{p}}\psi(t)dt\\
&\leq&\int_{0}^{1} (\|f\chi_{j-m}\|_{p,k}+\|f\chi_{j-m+1}\|_{p,k}
)\psi(t)dt.
\end{eqnarray*}
Using (4.1), it follows that\\
$\displaystyle{\Big(\displaystyle\sum_{j=-\infty}^{+\infty}2^{j\beta
q}\|(\mathcal{C_\varphi}f)\chi_j\|^{^q}_{p,k}\Big)^{\frac{1}{q}}}$
\begin{eqnarray*}
&\leq& \Big[\displaystyle\sum_{j=-\infty}^{+\infty}2^{j\beta
q}\Big(\int_{0}^{1} (\|f\chi_{j-m}\|_{p,k}+\|f\chi_{j-m+1}\|_{p,k}
)\,\psi(t)dt\Big)^{q}\Big]^{\frac{1}{q}}\\ &\leq&
\Big(\displaystyle\sum_{j=-\infty}^{+\infty}2^{j\beta q}\int_{0}^{1}
(\|f\chi_{j-m}\|_{p,k}+\|f\chi_{j-m+1}\|_{p,k}
)^{q}\,\psi^q(t)dt\Big)^{\frac{1}{q}},\end{eqnarray*}which
gives\\
$\displaystyle{\Big(\displaystyle\sum_{j=-\infty}^{+\infty}2^{j\beta
q}\|(\mathcal{C_\varphi}f)\chi_j\|^{^q}_{p,k}\Big)^{\frac{1}{q}}}$\begin{eqnarray*}
&\leq&
2^{1-\frac{1}{q}}\Big(\displaystyle\sum_{j=-\infty}^{+\infty}\int_{0}^{1}2^{j\beta
q} (\|f\chi_{j-m}\|^{q}_{p,k}+\|f\chi_{j-m+1}\|^{q}_{p,k})\,\psi^q(t)dt\Big)^{\frac{1}{q}}\nonumber\\
&\leq&2^{1-\frac{1}{q}}\Big(\int_{0}^{1}\displaystyle\sum_{j=-\infty}^{+\infty}2^{(j-m)\beta
q}\|f\chi_{j-m}\|^{q}_{p,k}2^{m\beta
q}\psi^q(t)dt\Big)^{\frac{1}{q}} \\
&&+\,2^{1-\frac{1}{q}}\Big(\int_{0}^{1}\displaystyle\sum_{j=-\infty}^{+\infty}2^{(j-m+1)\beta
q}\|f\chi_{j-m+1}\|^{q}_{p,k}2^{(m-1)\beta
q}\psi^q(t)dt\Big)^{\frac{1}{q}}.
\end{eqnarray*}
Then from (4.2), we have\\
$\displaystyle{\Big(\displaystyle\sum_{j=-\infty}^{+\infty}2^{j\beta
q}\|(\mathcal{C_\varphi}f)\chi_j\|^{^q}_{p,k}\Big)^{\frac{1}{q}}}$
\begin{eqnarray*}
&\leq&2^{1-\frac{2}{q}}(1+\frac{1}{q})
\Big(\sum_{j=-\infty}^{+\infty}2^{j\beta q}\| f
\chi_{j}\|_{p,k}^{q}\Big)^{\frac{1}{q}}\int_{0}^{1}
(2^{m\beta}+2^{(m-1)\beta})
\psi(t)dt\nonumber\\
&\leq&\Big[2^{1-\frac{2}{q}}(1+\frac{1}{q}) (1+2^{\beta}
)\int_{0}^{1}t^{\beta}\psi(t)dt\Big]
\Big(\sum_{j=-\infty}^{+\infty}2^{j\beta q}\| f
\chi_{j}\|_{p,k}^{q}\Big)^{\frac{1}{q}}.
\end{eqnarray*}
Hence we deduce that
\begin{eqnarray*}
\|\mathcal{C_\varphi}\|\leq c_{q,\beta}\int_0^1
t^{\beta-(2\gamma+d)(1-\frac{1}{p})}\varphi(t)dt,
\end{eqnarray*}
where $c_{q,\beta}=2^{(1-\frac{2}{q})}(1+\frac{1}{q})(1+2^{\beta})$.

Conversely, assume that the generalized Ces\`{a}ro operator
$\mathcal{C_\varphi}$ is bounded, then for $f\in H^{\beta,k}_{p,q}$,
\begin{eqnarray} \Big(\sum_{j=-\infty}^{+\infty}2^{j\beta
q}\|(\mathcal{C}_{\varphi}f)\chi_{j}\|_{p,k}^{q}\Big)^{\frac{1}{q}}
&\leq&
\|\mathcal{C}_{\varphi}\|\Big(\sum_{j=-\infty}^{+\infty}2^{j\beta
q}\| f  \chi_{j}\|_{p,k}^{q}\Big)^{\frac{1}{q}}.
\end{eqnarray}
  For any $\varepsilon\in
(0,1),$ we set
$$f_\varepsilon(x)=
\left\{\begin{array}{lll}\|x\|^{-(\beta+\varepsilon+\frac{2\gamma+d}{p})}&&\mbox{if}\;\,\|x\|>1\,\\&&\\
0&&\mbox{otherwise},\end{array}\right.$$ then for
$j=0,-1,-2,\cdots,\,\,\|f_{\varepsilon}\chi_{j}\|_{p,k}=0,$ and for
$j\in\mathbb{N}\backslash\{0\},$ it follows from (2.1) and (2.2)
that
\begin{eqnarray*}
\|f_{\varepsilon}\chi_{j}\|^{p}_{p,k}&=&\int_{2^{j-1}\leq\|x\|\leq2^{j}}\|x\|^{-(\beta+\varepsilon+\frac{2\gamma+d}{p})p}\omega_{k}(x)dx\nonumber\\
&=&\int_{S^{d-1}}w_{k}(y)d\sigma(y)\int_{2^{j-1}}^{2^{j}}r^{-(\beta+\varepsilon)p-1}dr\nonumber\\
&=&C_{\varepsilon}\, 2^{-j(\beta+\varepsilon)p}.
\end{eqnarray*}
where $\;\;\;\;\displaystyle{
 C_{\varepsilon}=\frac{2^{(\beta+\varepsilon)p}-1}{(\beta+\varepsilon)p}\int_{S^{d-1}}w_{k}(y)d\sigma(y)=
\frac{2^{(\beta+\varepsilon)p}-1}{(\beta+\varepsilon)p}\;d_k.} $\\
Hence, it yields
\begin{eqnarray}\Big(\displaystyle\sum_{j=-\infty}^{+\infty}2^{j\beta
q}\|f_{\varepsilon}\chi_{j}\|^{q}_{p,k}\Big)^{\frac{1}{q}}&=&
 \Big(\displaystyle\sum_{j=1}^{+\infty}2^{j\beta
q}\|f_{\varepsilon}\chi_{j}\|^{q}_{p,k}\Big)^{\frac{1}{q}}\nonumber\\
&=&
C^{\,\frac{1}{p}}_{\varepsilon}\Big(\displaystyle\sum_{j=1}^{+\infty}2^{-j\varepsilon
q}\Big)^{\frac{1}{q}}\nonumber\\
&=&C^{\,\frac{1}{p}}_{\varepsilon}\frac{2^{-q\varepsilon}}{(1-2^{-q\varepsilon})^{\frac{1}{q}}},
\end{eqnarray}thus $f_\varepsilon\in H^{\beta,k}_{p,q}$.
Now, it's easy to see that
\begin{eqnarray*}
\mathcal{C_\varphi}f_{\varepsilon}(x)&=& \int_{0}^{\|x\|}
 \|\frac{x}{t}\|^{-(\beta+\varepsilon+\frac{2\gamma+d}{p})} t^{-(2\gamma+d)}\varphi(t)dt\nonumber\\
&=&\|x\|^{-(\beta+\varepsilon+\frac{2\gamma+d}{p})}\int_{0}^{\|x\|}t^{\beta+\varepsilon-(2\gamma+d)(1-\frac{1}{p})}\varphi(t)dt.
\end{eqnarray*}
Recalling that $\psi(t)=t^{-(2\gamma+d)(1-\frac{1}{p})}\varphi(t)$,
$t\in(0,1)$, we
get\\$\displaystyle\sum_{j=-\infty}^{+\infty}2^{j\beta
q}\|(\mathcal{C_\varphi}f_{\varepsilon})\chi_{j}\|_{p,k}^{q}$
\begin{eqnarray*}
&=&\displaystyle\sum_{j=-\infty}^{+\infty}2^{j\beta
q}\Big[\int_{A_j}
\Big(\|x\|^{-(\beta+\varepsilon+\frac{2\gamma+d}{p})}\int_{0}^{\|x\|}t^{\beta+\varepsilon}\psi(t)dt\Big)^{p}w_{k}(x)dx\Big]^{\frac{q}{p}}\\
&\geq&\displaystyle\sum_{j=-\infty}^{+\infty}2^{j\beta
q}\Big[\int_{\|x\|>1}\chi_j(x)\|x\|^{-(\beta+\varepsilon+\frac{2\gamma+d}{p})p}\Big(\int_{0}^{\|x\|}
t^{\beta+\varepsilon}\psi(t)dt\Big)^{p}w_{k}(x)dx\Big]^{\frac{q}{p}}\\
&\geq&\Big(\int_{0}^{1}t^{\beta+\varepsilon}\psi(t)dt\Big)^{q}\displaystyle\sum_{j=1}^{+\infty}2^{j\beta
q}\Big(\int_{A_j}
\|x\|^{-(\beta+\varepsilon+\frac{2\gamma+d}{p})p}w_k(x)dx\Big)^{{\frac{q}{p}}}.
\end{eqnarray*}
It follows from (4.5) that
\begin{eqnarray*}
\Big(\sum_{j=-\infty}^{+\infty}2^{j\beta
q}\|(\mathcal{C_\varphi}f_{\varepsilon})\chi_{j}\|_{p,k}^{q}\Big)^{\frac{1}{q}}
&\geq&C^{\,\frac{1}{p}}_{\varepsilon}
\frac{2^{-\varepsilon}}{(1-2^{-q\varepsilon})^{\frac{1}{q}}}\Big(\int_{0}^{1}t^{\beta+\varepsilon}\psi(t)dt\Big),
\end{eqnarray*}
thus combining this inequality with (4.4) and (4.5), we can assert
that
\begin{eqnarray*}
\|\mathcal{C_\varphi}\|
\geq\int_{0}^{1}t^{\beta+\varepsilon-(2\gamma+d)(1-\frac{1}{p})}\varphi(t)dt,
\end{eqnarray*}
which implies when $\varepsilon\rightarrow 0$
\begin{eqnarray*}
\|\mathcal{C_\varphi}\|\geq\int_{0}^{1}t^{\beta-(2\gamma+d)(1-\frac{1}{p})}\varphi(t)dt.
\end{eqnarray*}
This completes the proof of the theorem.
\end{proof}


\begin{thebibliography}{}
%
%
\bibitem{} C. Abdelkefi and M.Sifi, Further results of integrability for the Dunkl
transform. Communication in Mathematical Analysis Vol. 2, N.1
(2007), 29-36.
\bibitem{} C. Abdelkefi, Dunkl transform on Besov spaces and Herz spaces.
Communication in Mathematical Analysis, Vol. 2, No 2 (2007), 35-41.
\bibitem{} C. Abdelkefi, J. Ph. Anker, F. Sassi and M. Sifi,
 Besov-type spaces on $\mathbb{R}^d $ and integrability for the Dunkl
transform. Symmetry, Integrability and Geometry: Methods and
Applications, SIGMA 5 (2009), 019, 15 pages.
\bibitem{} C. Abdelkefi, Dunkl
operators on $\mathbb{R}^d$ and uncentered maximal function.
 J. Lie Theory 20 (2010), No.1, 113-125.\bibitem{}
 C. Abdelkefi, Weighted function spaces and Dunkl
transform. Mediterr. J. Math. 9 (2012), 499-513 Springer Basel
AG.\bibitem{} E. F. Beckenbach and R. Bellman, Inequalities,
Springer-Verlag, Berlin, 1983.\bibitem{} O. V. Besov, On a family of
function spaces in connection with embeddings and extentions, Trudy
Mat. Inst. Steklov 60 (1961), 42-81.
\bibitem{} M. F. E. de Jeu , The Dunkl transform. Inv. Math, 113
(1993), 147-162.  \bibitem{} C. F. Dunkl, Differential-difference
operators associated to reflection groups. Trans. Amer. Math. Soc.
311, No1, (1989), 167-183.
\bibitem{} C. F. Dunkl, Integral kernels with reflection group
invariance. Can. J. Math. 43, No 6, (1991), 1213-1227.   \bibitem{}
C. F. Dunkl and Y. Xu, Orthogonal polynomials of several variables.
Cambridge Univ. Press, Cambridge, 2001. \bibitem{} C. S. Herz,
Lipschitz spaces and Bernstein's theorem on absolutely convergent
Fourier transform, J. Math. Mech. 18 (1968/69), 283-323.
\bibitem{} F. M\'{o}ricz, Sufficient conditions for the Lebesgue
integrability of Fourier transforms. Analysis Mathematica, 36
(2010), 121-129.
\bibitem{}
M. R\"osler and M. Voit, Markov processes with Dunkl operators. Adv.
in Appl. Math. 21, (1998), 575-643.\bibitem{} M. R\"osler, Dunkl
operators: theory and applications, in Orthogonal polynomials and
special functions (Leuven, 2002). Lect. Notes Math. 1817,
Springer-Verlag (2003), 93-135.\bibitem{} M. R\"osler, A positive
radial product formula for the Dunkl kernel. Trans. Amer. Math. Soc.
335, n°6, (2003), 2413-2438. \bibitem{} S. Thangavelyu and Y. Xu,
Convolution operator and maximal function for Dunkl transform. J.
Anal. Math. Vol. 97, (2005), 25-56. \bibitem{} K. Trim\`eche,
Paley-Wiener theorems for the Dunkl transform and Dunkl translation
operators. Integral Transforms Spec. Funct. 13, (2002),
17-38.\bibitem{} J.F. van Diejen, L. Vinet,
Calogero-Sutherland-Moser Models. CRM Series in Math. Phys.,
Springer-Verlag, 2000.
\end{thebibliography}
\end{document}